\documentclass[12pt,leqno]{article} 
\usepackage{amsmath,amstext, amsthm, amssymb}
  \numberwithin{equation}{section} 
\usepackage{epsfig}
\newtheorem{thm}{Theorem}[section]
\newtheorem{lem}[thm]{Lemma}

\newtheorem{cor}[thm]{Corollary}

\newcommand{\be}{\begin{equation}}
\newcommand{\ee}{\end{equation}}

\newcommand{\ba}{\begin{array}}
\newcommand{\ea}{\end{array}}

\newcommand{\bg}{\begin{gathered}}
\newcommand{\eg}{\end{gathered}}

\newcommand{\RR}{Rogers--Ramanujan}

\newcommand{\gauss}[2]{\genfrac{[}{]}{0pt}{}{#1}{#2}_q}

\newcommand{\bea}{\begin{eqnarray}}
\newcommand{\eea}{\end{eqnarray}}

\newcommand{\Sum}{\sum_{n=0}^\infty}

\begin{document}

\title{$q$-Bessel Functions and Rogers-Ramanujan Type Identities}
\author{ Mourad E. H.  Ismail 
 \thanks{Research partially supported  by the DSFP of King Saud University} \\
 Department of Mathematics, King Saud University\\
 Riyadh, Saudi Arabia  \\ and Department of Mathematics, University 
 of Central Florida, \\
 Orlando, Florida 32816, USA\\
 \and Ruiming Zhang \thanks{Research partially supported by National Science Foundation of China, grant No. 11371294.} \\
College of Science\\
Northwest A\&F University\\
Yangling, Shaanxi 712100\\
P. R. China.}

\maketitle

\begin{abstract}
 We evaluate $q$-Bessel functions at an infinite sequence of points 
 and introduce a generalization of the Ramanujan function and give 
 an extension of the $m$-version of the Rogers-Ramanujan identities.  
 We also prove several generating functions for Stieltjes-Wigert 
 polynomials with argument depending on the degree. In addition 
 we give  several  Rogers-Ramanujan type identities.  
 \end{abstract}
{\bf Mathematics Subject Classification MSC 2010}: Primary 11P84, 33D45 Secondary  05A17. 

\bigskip

\noindent{\bf filename}: Ism-ZhanRR2

\bigskip

\section{Introduction}
The  Rogers--Ramanujan identities  are 
\bea
\bg
\Sum \frac{q^{n^2}}{(q;q)_n} = \frac{1}{(q, q^4;q^5)_\infty} \\
\Sum \frac{q^{n^2+n}}{(q;q)_n} = \frac{1}{(q^2, q^3;q^5)_\infty},
 \eg
 \label{eqRR}
 \eea
 where the notation for the $q$-shifted factorials is the standard notation in \cite{Gas:Rah}, 
 \cite{Ismbook}.  References for the Rogers-Ramanujan identities, their origins and many of 
 their applications are in \cite{And2}, \cite{And}, and \cite{And:Ask:Roy}.  In particular we recall  
 the partition theoretic interpretation of the first Rogers--Ramanujan identity as the partitions of an 
 integer $n$ into parts $\equiv 1 \textup{or}\; 4 (\textup{mod} \; 5)$ are equinumerous with the part ions of $n$ 
 into parts where any two parts differ by at least 2. 
 
 Garrett, Ismail, and Stanton 
 \cite{Gar:Ism:Sta} proved the $m$-version of the Rogers-Ramanujan identities
 \bea
 \Sum \frac{q^{n^2+mn}}{(q;q)_n} =\frac{(-1)^m q^{-\binom{m}{2}} a_m(q)}{(q,q^4;q^5)_\infty}
- \frac{(-1)^{m} q^{-\binom{m}{2}} b_m(q)}{(q^2,q^3;q^5)_\infty},
\label{eqmform}
 \eea
 where
 \begin{equation}
\label{eq13.6.1}
\begin{gathered}
a_m(q)=\sum_{j} q^{j^2+j}\bmatrix m-j-2 \\ j\endbmatrix_q, \\
b_m(q) =\sum_{j} q^{j^2}\bmatrix m-j-1 \\ j\endbmatrix_q.
\end{gathered}
\end{equation}
The polynomials $a_m(q)$ and $b_m(q)$ were considered by Schur in conjunction with his proof of the  
Rogers--Ramanujan identities, see
\cite{And2}  and \cite{Gar:Ism:Sta}   for details. We shall refer to $a_m(q)$ and $b_m(q)$ as the Schur polynomials. The closed form expressions for $a_m$ and $b_m$ in  \eqref{eq13.6.1}  were given by 
 Andrews in \cite{And4}, where he also  
gave a polynomial generalization of the Rogers--Ramanujan identities. 
In this paper we give a family of Rogers--Ramanujan type identities involving the evaluation of $q$-Bessel and allied functions at special points. We also give the partition theoretic interpretation of these 
identities. In Section 2 we define the functions and polynomials used in our analysis. In Section 3 we present our \RR type identities. They resemble the $m$ form in \eqref{eqmform}.

In a series of papers from 1903 till 1905 F. H. Jackson  introduced $q$-analogues of 
Bessel functions. The modern notation for the modified $q$-Bessel functions, that is $q$-Bessel functions with imaginary argument, is, \cite{Ism82},
\begin{eqnarray}
I_\nu^{(1)}(z;q) &=&  \frac{(q^{\nu+1};q)_\infty} {(q;q)_\infty} 
\Sum \frac{   (z/2)^{\nu+2n}}{(q, q^{\nu+1};q)_n}, \quad    |z| < 2, 
\label{eqInu1}\\
I_\nu^{(2)}(z;q) &=& \frac{(q^{\nu+1};q)_\infty} {(q;q)_\infty}  \Sum \frac{  q^{n(n+\nu)}}{(q, q^{\nu+1};q)_n} (z/2)^{\nu+2n},\\
I_\nu^{(3)}(z;q) &=& \frac{(q^{\nu+1};q)_\infty} {(q;q)_\infty}  \Sum \frac{ q^{\binom{n}{2}}}
 {(q, q^{\nu+1};q)_n} (z/2)^{\nu+2n}. 
\label{eqInu3}
\end{eqnarray}
  The functions  $I_\nu^{(1)}$ and $I_\nu^{(2)}$ are related via 
  \begin{eqnarray}
  \label{eqIni-Inu2}
  I_\nu^{(1)}(z;q) =\frac{ I_\nu^{(2)}(z;q)}{(z^2/4;q)_\infty},  
  \end{eqnarray}
  \cite[Theorem 14.1.3]{Ismbook}.   
Formula \eqref{eqIni-Inu2} analytically continues $I_\nu^{(1)}$ to a 
meromorphic  function in the complex plane.  
The Stieltjes--Wigert polynomials \cite{Ismbook}, \cite{Sze},
are defined by 
\begin{eqnarray}
\label{eq:stieltjes1}
\begin{gathered}
\quad S_n(x;q)  =  \frac{1}{(q;q)_{n}}\sum_{k=0}^{n} 
\gauss{n}{k} q^{k^{2}}\left(-x\right)^{k}  
  =  \frac{1}{(q;q)_{n}}\sum_{k=0}^{n}\frac{(q^{-n};q)_{k}}{(q;q)_{k}}
  q^{\binom{k+1}{2}}\left(xq^{n}\right)^{k},   
\end{gathered}
\end{eqnarray}
respectively. Ismail and C. Zhang \cite{Ism:ZhaC} proved the following symmetry relation 
for the Stieltjes--Wigert polynomials
\bea
\label{eqsym}
q^{n^2}(-t)^nS_n(q^{-2n}/t;q)= S_n(t;q).
\eea

Section 2 contanis the evaluation of $I_\nu^{(2)}$ at an infinite 
number of special points. These new sums seem to be new. In 
Section 3 we introduce a generalization of the Ramanujan function  
\bea
\label{eqdefAq}
A_q(z) := \sum_{n=0}^\infty \frac{(-z)^n}{(q;q)_n} q^{n^2},
\eea
which  S. Ramanujan introduced and studied many of its properties 
In the lost note book \cite{Ram}. 
It was later realized that this is an analogue of the Airy function. In 
Section 4 we introduce a function $B^\alpha_q$ prove some identities 
it satisfies then use them to derive several Rogers-Ramanujan type 
identities. The function $B^\alpha_q$ is also a generalization of the 
Ramanujan function and is expected to lead to numerous new 
Rogers-Ramanujan type identities. The Stieltjes-Wigert polynomials 
satisfy a second order $q$-difference equation of polynomial 
coefficients of the for 
\bea
\notag
f(x) y(qx) + g(x) y(x) + h(x) y(x/q) =0. 
\eea
In Section 5 we evaluate $y(q^nx)$ in terms of $y(x$ and $y(x/q)$ 
with explicit coefficients. Section 6 contains misclaneous properties 
of the Stieltjes-Wigert polynomials. 

\section{$q$-Bessel Sums}
Our first result is the following theorem.
\begin{thm}\label{SpecialV}
The function $I_{\nu}^{(2)}$ has the represetation  
\bea
I_{\nu}^{(2)}\left(2z;q\right) 
= \frac{z^\nu}{(q;q)_\infty} \; {}_1\phi_1(z^2; 0; q, q^{\nu+1}).
\label{eq:specialvalugeneral}
\eea
In particular $I_{\nu}^{(2)}$  takes the special values  
\begin{equation}
I_{\nu}^{(2)}\left(2q^{-n/2};q\right)=\frac{q^{\nu n/2}S_{n}\left(-q^{-\nu-n};q\right)}{\left(q^{n+1};q\right)_{\infty}},\label{eq:special value bessel 4}
\end{equation}
 and 
\begin{equation}
I_{\nu}^{(2)}\left(2 q^{-n/2};q\right)=\frac{ q^{-\nu n/2}S_{n}\left(-q^{\nu-n};q\right)}{\left(q^{n+1};q\right)_{\infty}}\label{eq:special value bessel 5}
\end{equation}
\end{thm}
\begin{proof} 
Recall  the Heine transformation \cite[(III.2)]{Gas:Rah}   
\bea
{}_2\phi_1  \left(\left. \begin{matrix} 
 A , B \\
C
\end{matrix}\, \right|q,Z\right) 
&=& \frac{(C/B, BZ;q)_\infty}{(C, Z;q)_\infty} 
\; {}_2\phi_1  \left(\left. \begin{matrix} 
 ABZ/C , B \\
BZ
\end{matrix}\, \right|q,\frac{C}{B}\right).  
\label{eqHeine2}
\eea
The left-hand side of  \eqref{eq:specialvalugeneral}  is 
\bea
\notag
\begin{gathered}
\frac{(q^{\nu+1};q)_\infty}{(q;q)_\infty}   z^\nu \sum_{k=0}^\infty
\frac{q^{k^2+k\nu}z^{2k}}{(q^{\nu+1}, q;q)_k} 
=
 \frac{(q^{\nu+1};q)_\infty}{(q;q)_\infty} z^\nu \lim_{a, b \to \infty}
 {}_2\phi_1  \left(\left. \begin{matrix} 
 a, b \\
q^{\nu+1}
\end{matrix}\, \right|q,\frac{q^{\nu+1}z^2}{ab} \right) \\
= \frac{(q^{\nu+1};q)_\infty}{(q;q)_\infty} z^{\nu} \frac{1}{(q^{\nu+1};q)_\infty} 
\; \lim_{a, b \to \infty}
 {}_2\phi_1  \left(\left. \begin{matrix} 
 z^2 , b \\
z^2 q^{\nu+1}/a
\end{matrix}\, \right|q,\frac{q^{\nu+1}}{b} \right) 
\end{gathered}
\eea
which implies \eqref{eq:specialvalugeneral}.  When $z = q^{-n/2}$ and  in view of \eqref{eq:stieltjes1},  the left-hand side of  \eqref{eq:special value bessel 4} 
equals its right-hand side.  
Formula \eqref{eq:special value bessel 5} follows from the symmetry relation \eqref{eqsym}
\end{proof}
 
 The results \eqref{eq:special value bessel 4}--\eqref{eq:special value bessel 5} of Theorem \ref{SpecialV} when written as a series becomes 
 \bea
 \bg
 \sum_{k=0}^\infty \frac{q^{k(k+\nu-n)}}{(q, q^{\nu+1};q)_k} 
 = \frac{q^{n\nu}}{(q^{\nu+1};q)_\infty}  \sum_{k=0}^n \gauss{n}{k} q^{k^2} q^{-k(\nu+n)}
 \\ = \frac{1}{(q^{\nu+1};q)_\infty}  \sum_{k=0}^n \gauss{n}{k} q^{k^2} q^{k(\nu-n)}. 
 \eg
 \eea

Another way to prove \eqref{eq:special value bessel 4} for integer $\nu$  is to use the generating function 
\bea
\sum_{-\infty}^\infty q^{\binom{m}{2}} I_m^{(2)}(z;q) t^m = (-tz/2, -qz/2t;q)_\infty.
\label{eqgfqBI}
\eea
Carlitz \cite{Car} did this for $n=0,1$ and used this to give another proof of the Rogers--Ramanujan identities.

\begin{thm}\textup{\cite{And2}}
The $q$-binomial coefficient  $\gauss{n}{k}$ is the generating  function  for integer partitions whose Ferrers diagrams fit inside a $k \times (n-k)$ rectangle.
 \end{thm} 

Recall that 
\bea
\label{eqIvsJ}
I_\nu^{(j)}(z;q) = e^{-i\nu \pi/2} J_\nu^{(j)}( e^{i\pi/2}z;q), j=1,2.
\eea     
Chen, Ismail, and Muttalib \cite{Che:Ism:Mut} established an asymptotic series for $J_\nu^{(2)}(z;q)$. Their 
main term for $r> 0$  is 
\bea
\label{eqasJnu}
\bg
I_\nu^{(2)}(r;q) = (r/2)^\nu \frac{(q^{1/2};q)_\infty}{2(q;q)_\infty}  
\qquad    \qquad \qquad \qquad  \\
\qquad \qquad  \times 
\left[(rq^{(\nu+1/2)/2}/2;q^{1/2})_\infty + (-rq^{(\nu+1/2)/2}/2;q^{1/2})_\infty, 
\right]
\eg
\eea
as $r \to +\infty$. This determines the large $r$ behavior of the maximum modulus of $I_\nu^{(2)}$. 

We next derive a Mittag--Leffler expansion for $I_{\nu}^{(1)}$. 

\begin{thm} \label{Mittag}
We have the expansion 
\begin{equation}
I_{\nu}^{(1)}\left(z;q\right)=\frac{\left(\frac{z}{2}\right)^{\nu}}
{\left(q;q\right)_{\infty}^{2}}\sum_{n=0}^{\infty}
\frac{\left(-1\right)^{n}q^{\binom{n+1}{2}}S_{n}\left(-q^{\nu-n};q\right)}
{\left(1-z^{2}q^{n}/4\right)}.\label{eq:special value bessel 12}
\end{equation}
\end{thm}
Using residues it is easy to see that the difference 
between  $I_{\nu}^{(1)}\left(z;q\right)/(z^2;q)_\infty$ and the right-hand 
side of \eqref{eq:special value bessel 12}
is entire.  We give a direct proof that this difference is zero.
\begin{proof}[Proof of Theorem \ref{Mittag}]
Use  \eqref{eq:stieltjes1} to see that the sum on the right-hand side of 
\eqref{eq:special value bessel 12} is 
\bea
\notag
\bg
\sum_{n=0}^{\infty}
\frac{\left(-1\right)^{n}q^{\binom{n+1}{2}}}
{\left(1-z^{2}q^{n}/4\right)}\sum_{k=0}^n \frac{q^{k^2+k(\nu-n)}}
{(q;q)_k(q;q)_{n-k}} \qquad \qquad \qquad \qquad \\
=  \sum_{k=0}^\infty \frac{(-1)^k q^{k(\nu+(k+1)/2)}}
{(q;q)_k(1-z^2q^k/4)}\; {}_1\phi_1(z^2q^k/4; z^2q^{k+1}/4;q, q). 
\eg
\eea
Now apply (III.4) of \cite{Gas:Rah} with $a = z^2q^k/4, b=1, c =0, z=q$ 
to see that the above sum is  $(q;q)_\infty/(z^2q^{k+1}/4;q)_\infty$. 
This shows that the right-hand side of \eqref{eq:special value bessel 12} 
is given by 
\bea
\notag
\frac{(z/2)^\nu}{(q, z^2/4;q)_\infty} \; {}_1\phi_1(z^2/4; 0; q, q^{\nu+1}),
\eea
and the result follows from \eqref{eq:specialvalugeneral} and 
\eqref{eqIni-Inu2}. 
\end{proof}

\setcounter{equation}{0}
\section{A Generalization of the Ramanujan Function}
 The 
Rogers-Ramanujan identities evaluate $A_q$ at $z =-1, -q$, The result 
\eqref{eqmform} evaluates $A_q(-q^m)$. This motivated us to 
 consider the function
\bea
u_m(a, q) := \sum_{n=-\infty}^\infty \frac{q^{n^2+mn}}{(aq;q)_n}, 
\eea
as a function of $q^m$. 
When $a =1$ we get the Rogers-Ramanujan  function. It is clear that 
\bea
q^{m+1}u_{m+2}(a, q) =\sum_{n=-\infty}^\infty \frac{(1-aq^n)}{(aq;q)_n}q^{n^2+mn}
\nonumber
\eea
Therefore 
\bea
q^{m+1} u_{m+2}(a, q) = u_m(a,q) -a u_{m+1}(a,q). 
\eea 
Let $u_m(a,q) = q^{-\binom{m}{2}}(-1)^m \, \tilde{u}_m(a,q)$. Then 
$\{\tilde{u}_m(a,q)\}$ satisfy the difference equation 
\bea
\label{eqdiffeqY}
y_{m+1} = q^{m-1} \, y_{m-1} + a y_m, \quad m >0.
\eea
We now solve \eqref{eqdiffeqY} using generating functions. The generating function $Y(t) :=\Sum y_n t^n$ satisfies 
$$
Y(t) = \frac{y_0+t(y_1-ay_0)}{1-at} \, + \frac{t^2}{1-at}\; Y(qt),
$$
whose solution is 
\bea
\notag
Y(t) = \Sum \frac{q^{n(n-1)}t^{2n}}{(at;q)_{n+1}}\; [y_0+tq^n(y_1-ay_0]. 
\eea
We now need two initial conditions, so choose two solutions 
$\{c_m(a,q)\}$ 
and $\{d_m(a,q)\}$
\bea
\label{eqIC}
c_0(a,q) = 1, c_1(a, q) = 0, \;\;
c_0(a,q) = 0, d_1(a, q) = 1. 
\eea

\begin{thm}
The polynomials $\{c_m(a,q)\}$ 
and $\{d_m(a,q)\}$ have the generating functions 
\bea
\label{eqgecn}
\Sum  c_n(a,q) t^n &=& \Sum \frac{q^{n(n-1)}}{(at;q)_n}\, t^{2n}, 
\\ 
 \Sum   d_n(a,q) t^n &=& \Sum \frac{q^{n^2}t^{2n+1}}{(at;q)_{n+1}}.
\label{eqgedn}
\eea
\end{thm}

It is clear from the initial conditions \eqref{eqIC} and the recurrence 
relation  \eqref{eqdiffeqY} that both $\{c_n(a,q)\}$ and $\{d_n(a,q)\}$ 
are polynomials in   $a$ and in $q$. 
\begin{thm}
The polynomials  $\{c_n(a,q)\}$ and $\{d_n(a,q)\}$ have the explicit form 
\bea
\label{eqcn}
c_n(a, q) &=& \sum_{j=0}^{\lfloor(n-2)/2\rfloor} q^{j(j+1)}\; \gauss{n-j-2}{j} 
a ^{n-2j-2},  \\
d_n(a,q) &=& \sum_{j=0}^{\lfloor(n-1)/2\rfloor} q^{j^2}\; \gauss{n-j-1}{j}
a ^{n-2j-1}.
\label{eqdn} 
\eea
\end{thm}
 The proof follows form equations \eqref{eqgecn} and \eqref{eqgedn};   and the 
 $q$-binomial theorem. 
\begin{thm}
We have 
\bea
\notag
\begin{gathered}
\sum_{n = -\infty}^\infty \frac{q^{n^2+mn}}{(aq;q)_n} 
= (-1)^m q^{-\binom{m}{2}}\\
\times 
 \left[c_m(a,q) \sum_{n = -\infty}^\infty \frac{q^{n^2}}{(aq;q)_n}
 +  d_m(a,q) \sum_{n = -\infty}^\infty \frac{q^{n^2+n}}{(aq;q)_n} \right]
\end{gathered}
\eea
\end{thm}

The case $a = 1$ is the $m$-version of the Rogers-Ramanujan 
identities in \eqref{eqmform} first proved by Garret, Ismail, and Stanton 
\cite{Gar:Ism:Sta}.

\section{Rogers-Ramanujan Type Identities}
In this section we prove several identities of Rogers-Ramanujan type. 
One of the proofs uses the Ramanujan ${}_1\psi_1$ sum 
\cite[(II.29)]{Gas:Rah}
\bea
\label{eq1psi1}
\sum_{-\infty}^\infty \frac{(a;q)_n}{(b;q)_n} z^n = 
\frac{(q, b/a, az, q/az;q)_\infty}{(b, q/a, z, b/az;q)_\infty}, \quad \left|\frac{b}{a}\right| < |z| < 1. 
\eea

Throughout this section we define $\rho$ by 
\bea
\rho=e^{2\pi i/3}. 
\label{eqdefrho}
\eea
\begin{lem}
For nonnegative integer $j,k,\ell,m,n$ and $\rho=e^{2\pi i/3}$ we
have 
\begin{equation}
\sum_{k=0}^{n}\frac{\left(a;q\right)_{k}\left(a;q\right)_{n-k}\left(-1\right)^{k}}{\left(q;q\right)_{k}\left(q;q\right)_{n-k}}=\begin{cases}
0 & n=2m+1\\
\frac{\left(a^{2};q^{2}\right)_{m}}{\left(q^{2};q^{2}\right)_{m}} & n=2m
\end{cases},\label{eq:multiple sum 1}
\end{equation}
 and
\begin{equation}
\sum_{\begin{array}{c}
j+k+\ell=n\\
j,k,\ell\ge0
\end{array}}\frac{\left(a;q\right)_{j}\left(a;q\right)_{k}\left(a;q\right)_{\ell}}{\left(q;q\right)_{j}\left(q;q\right)_{k}\left(q;q\right)_{\ell}}\rho^{k+2\ell}=\begin{cases}
0 & 3\nmid n\\
\frac{\left(a^{3};q^{3}\right)_{m}}{\left(q^{3};q^{3}\right)_{m}} & n=3m
\end{cases}.\label{eq:multiple sum 2}
\end{equation}
 For $j,k,m,\ell,n\in\mathbb{Z}$, we have 
\begin{equation}
\sum_{j+k=n}\frac{\left(a;q\right)_{j}\left(a;q\right)_{k}\left(-1\right)^{k}}{\left(b;q\right)_{j}\left(b;q\right)_{k}}=\begin{cases}
0 & n=2m+1\\
\frac{\left(q,b/a,-b,-q/a;q\right)_{\infty}}{\left(-q,-b/a,b,q/a;q\right)_{\infty}}\frac{\left(a^{2};q^{2}\right)_{m}}{\left(b^{2};q^{2}\right)_{m}} & n=2m
\end{cases}\label{eq:multiple sum 3}
\end{equation}
 and 
\begin{equation}
\sum_{j+k+\ell=n}^{\infty}\frac{\left(a;q\right)_{j}\left(a;q\right)_{k}\left(a;q\right)_{\ell}\rho^{k+2\ell}}{\left(b;q\right)_{j}\left(b;q\right)_{k}\left(b;q\right)_{\ell}}=0\label{eq:multiple sum 4}
\end{equation}
 for $3\nmid n$, 
\begin{eqnarray}
 &  & \sum_{j+k+\ell=3m}^{\infty}\frac{\left(a;q\right)_{j}\left(a;q\right)_{k}\left(a;q\right)_{\ell}\rho^{k+2\ell}}{\left(b;q\right)_{j}\left(b;q\right)_{k}\left(b;q\right)_{\ell}}\label{eq:multiple sum 5}\\
 & = & \frac{\left(q,b/a;q\right)_{\infty}^{3}}{\left(b,q/a;q\right)_{\infty}^{3}}\frac{\left(b^{3},q^{3}a^{-3};q^{3}\right)_{\infty}}{\left(q^{3},b^{3}a^{-3};q^{3}\right)}\frac{\left(a^{3};q^{3}\right)_{m}}{\left(b^{3};q^{3}\right)_{m}}.\nonumber 
\end{eqnarray}
 \end{lem}
\begin{proof} Formula \eqref{eq:multiple sum 1} follows from
\[
\frac{\left(at;q\right)_{\infty}}{\left(t;q\right)_{\infty}}\frac{\left(-at;q\right)_{\infty}}{\left(-t;q\right)_{\infty}}=\frac{\left(a^{2}t^{2};q^{2}\right)_{\infty}}{\left(t^{2};q^{2}\right)_{\infty}}, 
\quad \left|t\right|<1, 
\]
while \eqref{eq:multiple sum 2} follows from 
\[
\frac{\left(at;q\right)_{\infty}}{\left(t;q\right)_{\infty}}\frac{\left(a\rho t;q\right)_{\infty}}{\left(\rho t;q\right)_{\infty}}\frac{\left(a\rho^{2}t;q\right)_{\infty}}{\left(\rho^{2}t;q\right)_{\infty}}=\frac{\left(a^{3}t^{3};q^{3}\right)}{\left(t^{3};q^{3}\right)}, \quad \left|t\right|<1. 
\]
For $\left|ba^{-1}\right|<\left|x\right|<1$, apply the Ramanujan
${}_{1}\psi_{1}$ sum \eqref{eq1psi1} to the identity 
\[
\frac{\left(ax,q/\left(ax\right);q\right)_{\infty}}
{\left(x,b/\left(ax\right);q\right)_{\infty}}
\frac{\left(-ax,-q/\left(ax\right);q\right)_{\infty}}
{\left(-x,-b/\left(ax\right);q\right)_{\infty}}
=\frac{\left(a^{2}x^{2},q^{2}/\left(a^{2}x^{2}\right);q^{2}\right)_{\infty}}
{\left(x^{2},b^{2}/\left(a^{2}x^{2}\right);q^{2}\right)_{\infty}},
\]
to derive \eqref{eq:multiple sum 3}. Similarly we apply \eqref{eq1psi1} to 
\begin{eqnarray*}
\bg
  \frac{\left(a^{3}x^{3},q^{3}/\left(a^{3}x^{3}\right);q^{3}\right)_{\infty}}{\left(x^{3},b^{3}/\left(a^{3}x^{3}\right);q^{3}\right)_{\infty}}\\
  =\frac{\left(ax\rho^{2},q/\left(ax\rho^{2}\right);q\right)_{\infty}}{\left(x\rho^{2},b/\left(ax\rho^{2}\right);q\right)_{\infty}}\;  \frac{\left(ax\rho,-q/\left(ax\rho\right);q\right)_{\infty}}{\left(x\rho,-b/\left(ax\rho\right);q\right)_{\infty}}\; \frac{\left(ax,q/\left(ax\right);q\right)_{\infty}}{\left(x,b/\left(ax\right);q\right)_{\infty}},
  \eg
\end{eqnarray*}
and establish \eqref{eq:multiple sum 4}-\eqref{eq:multiple sum 5}. 
\end{proof}
It must be noted that \eqref{eq:multiple sum 1} is essentially the evaluation of a continuous $q$-ultraspherical polynomial at $x =0$, 
\cite[(12.2.19)]{Ismbook}.

For $\alpha>0$, let
\begin{equation}
A_{q}^{\left(\alpha\right)}\left(a;t\right)=\sum_{n=0}^{\infty}\frac{\left(a;q\right)_{n}q^{\alpha n^{2}}t^{n}}{\left(q;q\right)_{n}},\label{eq:multiple sum 6}
\end{equation}
 in particular,
\[
A_{q}^{\left(1\right)}\left(q;t\right)=\omega\left(t;q\right),\quad A_{q^{2}}^{\left(2\right)}\left(q^{2};t^{2}\right)=\omega\left(t^{2};q^{4}\right),\quad A_{q}^{\left(1\right)}\left(0;t\right)=A_{q}\left(-t\right),
\]
where 
\[
\omega\left(v;q\right)=\sum_{n=0}^{\infty}q^{n^{2}}v^{n}.
\]

\begin{thm}
Let $\alpha\ge0$, then
\begin{equation}
A_{q^{2}}^{\left(2\alpha\right)}\left(a^{2};t^{2}\right)=\sum_{j=0}^{\infty}\frac{\left(a;q\right)_{j}q^{\alpha j^{2}}\left(-t\right)^{j}}{\left(q;q\right)_{j}}A_{q}^{\left(\alpha\right)}\left(a;tq^{2\alpha j}\right).\label{eq:multiple sum 7}
\end{equation}
For $\rho=e^{2\pi i/3}$ we have
\begin{equation}
A_{q^{3}}^{\left(3\alpha\right)}\left(a^{3};t^{3}\right)=\sum_{j,k=0}^{\infty}\frac{\left(a;q\right)_{j}\left(a;q\right)_{k}\rho^{k}q^{\alpha\left(j+k\right)^{2}}t^{j+k}}{\left(q;q\right)_{j}\left(q;q\right)_{k}}A_{q}^{\left(\alpha\right)}\left(a;\rho^{2}q^{2\alpha\left(j+k\right)}t\right).\label{eq:multiple sum 8}
\end{equation}
\end{thm}
\begin{proof}
These two identities can be proved by applying (\ref{eq:multiple sum 1})
and (\ref{eq:multiple sum 2}) and straightforward series manipulation.
\end{proof}

We now consider the following generalization of the ${}_1\psi_1$ function. 
For $\alpha\ge0$, define $B_{q}^{\left(\alpha\right)}$ by 
\begin{equation}
B_{q}^{\left(\alpha\right)}\left(a,b;x\right)=\sum_{n=-\infty}^{\infty}\frac{\left(a;q\right)_{n}}{\left(b;q\right)_{n}}q^{\alpha n^{2}}x^{n},\label{eq:multiple sum 10}
\end{equation}
\begin{thm} We have 
\begin{equation}
\begin{aligned}\frac{\left(-b,-q/a,q,b/a;q\right)_{\infty}}{\left(-q,-b/a,b,q/a;q\right)_{\infty}}B_{q^{2}}^{\left(2\alpha\right)}\left(a^{2},b^{2};x^{2}\right) & =\sum_{j=-\infty}^{\infty}\frac{\left(a;q\right)_{j}q^{\alpha j^{2}}\left(-x\right)^{j}}{\left(b;q\right)_{j}}B_{q}^{\left(\alpha\right)}\left(a,b;xq^{2\alpha j}\right).\end{aligned}
\label{eq:multiple sum 11}
\end{equation}
 and
\begin{equation}
\begin{aligned}B_{q^{3}}^{\left(3\alpha\right)}\left(a^{3},b^{3};x^{3}\right) & =\frac{\left(b,q/a;q\right)_{\infty}^{3}}{\left(q,b/a;q\right)_{\infty}^{3}}\frac{\left(q^{3},b^{3}a^{-3};q^{3}\right)}{\left(b^{3},q^{3}a^{-3};q^{3}\right)_{\infty}}\\
 & \times\sum_{j,k=-\infty}^{\infty}\frac{\left(a;q\right)_{j}\left(a;q\right)_{k}\rho^{k}q^{\alpha\left(j+k\right)^{2}}x^{j+k}}{\left(b;q\right)_{j}\left(b;q\right)_{k}}B_{q}^{\left(\alpha\right)}\left(a,b;xq^{2\alpha\left(j+k\right)}\right).
\end{aligned}
\label{eq:multiple sum 12}
\end{equation}
\end{thm}
The proof follows from   \eqref{eq:multiple sum 3}, 
\eqref{eq:multiple sum 4} and
\eqref{eq:multiple sum 5} and straightforward series manipulation. 

\begin{cor}
The following Rogers-Ramanujan type identities hold
\bea
\quad \frac{\left(-a,-q/a,q,q;q\right)_{\infty}}{\left(a,q/a,-q,-q;q\right)_{\infty}}
\sum_{n=-\infty}^{\infty}\frac{q^{4n^{2}}x^{2n}}{1-a^{2}q^{2n}}
&=&\sum_{j,k=-\infty}^{\infty}\frac{q^{\left(j+k\right)^{2}}
\left(-1\right)^{j}x^{j+k}}{\left(1-aq^{j}\right)\left(1-aq^{k}\right)},
\label{eq:multiple sum 13}\\
\frac{\left(q,q;q\right)_{\infty}}{\left(-q,-q;q\right)_{\infty}}
\sum_{n=-\infty}^{\infty}\frac{q^{4n^{2}}x^{2n}}{1+q^{2n+1}}
&=&\sum_{j,k=-\infty}^{\infty}\frac{q^{\left(j+k\right)^{2}}
\left(-1\right)^{j}x^{j+k}}{\left(1+iq^{j+1/2}\right)\left(1+iq^{k+1/2}\right)}.
\label{eq:multiple sum 14}
\eea
\end{cor}
\begin{proof}  Formula \eqref{eq:multiple sum 13} is the special case 
 $\alpha=1$ and $b=aq$ of (\ref{eq:multiple sum 11})  while 
 \eqref{eq:multiple sum 14} is the speical case  $a=-q^{1/2}i$ 
 of \eqref{eq:multiple sum 13}.
 \end{proof}
  
The special choice  $\alpha=1$ and $b=aq$ in (\ref{eq:multiple sum 12}) establishes 
\begin{equation}
\begin{aligned}\sum_{n=-\infty}^{\infty}\frac{q^{9n^{2}}x^{3n}}{1-a^{3}q^{3n}} & =\frac{\left(q^{3};q^{3}\right)_{\infty}^{2}}{\left(q;q\right)_{\infty}^{6}}\frac{\left(a,q/a;q\right)_{\infty}^{3}}{\left(a^{3},q^{3}a^{-3};q^{3}\right)_{\infty}}\\
 & \times\sum_{j,k,\ell=-\infty}^{\infty}\frac{\rho^{k+2\ell}q^{\left(j+k+\ell\right)^{2}}x^{j+k+\ell}}{\left(1-aq^{j}\right)\left(1-aq^{k}\right)\left(1-aq^{\ell}\right)}.
\end{aligned}
\label{eq:multiple sum 15}
\end{equation}
 Two special case   of \eqref{eq:multiple sum 15}  are worth noting.  First when  $a=q^{1/3}$ we find that
\begin{equation}
\begin{aligned} & \frac{\left(q;q\right)_{\infty}^{7}}{\left(q^{3};q^{3}\right)_{\infty}^{3}\left(q^{1/3},q^{2/3};q\right)_{\infty}^{3}}\sum_{n=-\infty}^{\infty}\frac{q^{9n^{2}}x^{3n}}{1-q^{3n+1}}\\
 & =\sum_{j,k,\ell=-\infty}^{\infty}\frac{\rho^{k+2\ell}q^{\left(j+k+\ell\right)^{2}}x^{j+k+\ell}}{\left(1-q^{j+1/3}\right)\left(1-q^{k+1/3}\right)\left(1-q^{\ell+1/3}\right)}.
\end{aligned}
\label{eq:multiple sum 16}
\end{equation}
With $a=-q^{1/3}$ in (\ref{eq:multiple sum 15}) we 
conclude that 
\begin{equation}
\begin{aligned}\sum_{n=-\infty}^{\infty}\frac{q^{9n^{2}}x^{3n}}
{1+q^{3n+1}} & =\frac{\left(q^{3};q^{3}\right)_{\infty}^{2}}
{\left(q;q\right)_{\infty}^{6}}
\frac{\left(-q^{1/3},-q^{2/3};q\right)_{\infty}^{3}}
{\left(-q^{2},-q;q^{3}\right)_{\infty}}\\
 & \times\sum_{j,k,\ell=-\infty}^{\infty}
 \frac{\rho^{k+2\ell}q^{\left(j+k+\ell\right)^{2}}x^{j+k+\ell}}
 {\left(1+q^{j+1/3}\right)\left(1+q^{k+1/3}\right)\left(1+q^{\ell+1/3}\right)}.
\end{aligned}
\label{eq:multiple sum 17}
\end{equation}

It is clear that one can generate other identities by specializing the parameters in the master formulas. 
  
 \section{$q$-Lommel Polynomials}
 Iterating the three term recurrence relation of the $q$-Bessel function leads to     
\begin{equation}
q^{n\nu+n(n-1)/2}J_{\nu+n}^{(2)}(x;q)=h_{n,\nu}\left(\frac{1}{x};q\right) 
J_{\nu}^{(2)}(x;q)-h_{n-1,\nu+1}\left(\frac{1}{x};q\right)
J_{\nu-1}^{(2)}(x;q), \label{eq:special value bessel 13}
\end{equation}
where  $h_{n,\nu}\left(x;q\right)$ are  the $q$-Lommel polynomials 
introduced in \cite{Ism82}, \cite[\S 14.4]{Ismbook}. 
It is more convenient to use the polynomials 
\bea
p_{n,\nu}(x;q) := e^{-i\pi n/2} h_{n,\nu}(ix) = \sum_{j=0}^{\lfloor{n/2}\rfloor} 
\frac{(q^\nu,q;q)_{n-j}}{(q, q^\nu;q)_j(q;q)_{n-2j}} (2x)^{n-2j} 
q^{j(j+\nu-1)}. 
\eea
The identity  \eqref{eq:special value bessel 13} expressed in terms of 
$I_\nu$'s is 
\bea
\label{eq3trrI}
\bg
 (-1)^n q^{n\nu+n(n-1)/2}I_{\nu+n}^{(2)}(x;q) \\
 =p_{n,\nu}(1/x;q) 
I_{\nu}^{(2)}(x;q)- p_{n-1,\nu+1}(1/x;q) I_{\nu-1}^{(2)}(x;q), 
\eg
\eea
When $x = 2q^{-k/2}$ we obtain, after replacing $\nu$ by $\nu+k$, 
\begin{eqnarray*}
\bg
 (-1)^n q^{n(n+2\nu+k-1)/2}S_{k}\left(-q^{\nu+n};q\right)  =  
p_{n,\nu+k}(q^{k/2}/2;q)S_{k}\left(-q^{\nu};q\right)\\
  -  q^{k/2}p_{n-1,\nu+k+1}(q^{k/2}/2;q)S_{k}\left(-q^{\nu-1};q\right). 
  \eg
\end{eqnarray*}
We now rewrite this as a functional equation in the form 
\begin{eqnarray}
\bg
 y^n  q^{n(n+k-1)/2}S_{k}\left(y q^{n};q\right)  =  
u_{n}(q^{k/2}, -yq^k;q)S_{k}\left(y;q\right)\\
  -  q^{k/2}u_{n-1}(q^{k/2}, -yq^{k+1};q)S_{k}\left(y/q;q\right). 
  \eg
\end{eqnarray}
with 
\bea
u_n(x,y) = \sum_{j=0}^{\lfloor{n/2}\rfloor} 
\frac{(y,q;q)_{n-j}}{(q, y;q)_j(q;q)_{n-2j}} x^{n-2j}.  
\eea

Therefore 
\begin{eqnarray}
\label{eq:special value bessel 15}
\bg
S_k(y;q) =  
\frac{y^n  q^{n(n+k-1)/2}u_{n}(q^{k/2}, -yq^{k+1};q)}
{\Delta_n}S_{k}(y q^{n};q) 
 \\
 -  \frac{y^{n+1}q^{(n+1) (n+ k)/2}
 u_{n+1}(q^{k/2}, -yq^{k+1};q)}{\Delta_{n}}
 S_{k}(-q^{\nu+n+1};q),
 \eg
\end{eqnarray}
where
\bea
\bg
\Delta_{n} =u_{n}(q^{k/2}, -yq^{k+1};q) u_{n}(q^{k/2}, -yq^{k};q) 
\qquad \qquad\\
 \qquad \qquad - u_{n+1}(q^{k/2}, -yq^{k};q) u_{n-1}(q^{k/2}, -yq^{k+1};q).
  \label{eq:special value bessel 16}
\eg
\eea

\section{Identities Involving Stieltjs--Wigert Polynomials}

In this section we state several identities involving Stieltjes--Wigert polynomials and the Ramanujan function.   
\bea
\left(xt,-t;q\right)_{\infty}&=&\sum_{n=0}^{\infty}q^{\binom{n}{2}}t^n S_{n}\left(xq^{-n};q\right). 
\label{eq:stieltjes 5.1}\\
\frac{q^{\binom{n}{2}}x^{n}}{\left(q;q\right)_{n}}&=&\sum_{k=0}^{n}
\frac{(-1)^kq^{\binom{k}{2}}}{\left(q;q\right)_{n-k}} 
S_{k}\left(xq^{-k};q\right),
\label{eq:stieltjes 5.2}\\
S_{n}\left(x\right)&=&\sum_{k=0}^{\infty}\frac{q^{\binom{k+1}{2}}
(xq^n)^{k}A_{q}\left(xq^{k}\right)}{\left(q;q\right)_n
\left(q;q\right)_{k}},
\label{eq:stieltjes 5.3}\\
S_{n}\left(ab;q\right)&=& b^{n}\sum_{k=0}^{n}\frac{\left(b^{-1};q\right)_{k}
\left(-q^{1-n}\right)^{k}q^{\binom{k}{2}}}{\left(q;q\right)_{k}}S_{n-k}
\left(aq^{k};q\right),
\label{eq:stieltjes 5.4}\\
S_{n}\left(a;q\right)&=&\frac{\left(-aq;q\right)_{\infty}}
{\left(q,-aq;q\right)_{n}}\sum_{k=0}^{\infty}\frac{q^{k^{2}}
\left(-a\right)^{k}}{\left(q,-aq^{n+1};q\right)_{k}},
\label{eq:stieltjes 5.5}\\
S_{2n+1}\left(q^{-2n-1};q\right)&=&0,\quad S_{2n}\left(q^{-2n};q\right)=\frac{\left(-1\right)^{n}q^{-n^{2}}}{\left(q^{2};q^{2}\right)_{n}}.\label{eq:stieltjes 5.6}\\
S_{n}\left(-q^{-n+1/2};q\right)&=&\frac{q^{-\left(n^{2}-n\right)/4}}{\left(q^{1/2};q^{1/2}\right)_{n}}, 
 \label{eq:stieltjes 5.7}\\
S_{n}\left(-q^{-n-1/2};q\right)&=&\frac{q^{-\left(n^{2}+n\right)/4}}{\left(q^{1/2};q^{1/2}\right)_{n}},\label{eq:stieltjes 5.8}
\\
A_{q}\left(wz\right)&=&\left(wq;q\right)_{\infty}\sum_{n=0}^{\infty}\frac{q^{n^{2}}w^{n}}{\left(wq;q\right)_{n}}S_{n}\left(zq^{-n};q\right).\label{eq:stieltjes 5.9}
\\
A_{q}\left(z\right)&=&\left(q;q\right)_{m}
\sum_{n=0}^{\infty}\frac{q^{n^{2}+mn}\left(-z\right)^{n}}{\left(q;q\right)_{n}}S_{m}\left(zq^{n};q\right).
\label{eq:stieltjes 10}
\eea

\begin{proof}[Proofs]
Formula \eqref{eq:stieltjes 5.1} follows from the definition 
\eqref{eq:stieltjes1} and Euler's identities. Dividing both sides of 
 \eqref{eq:stieltjes 5.1} by $(-t;q)_\infty$ then expand $1/(-t;q)_\infty$ on 
 the right-hand side implies  \eqref{eq:stieltjes 5.2}. The expansion 
 \eqref{eq:stieltjes 5.3} follows from \eqref{eqdefAq},  and the 
 $q$-binomial theorem in the form 
 \bea
 \label{eqfiniteqbinom}
 (x;q)_n = \sum_{j=0}^n \gauss{n}{j}  (-x)^j q^{\binom{k}{2}}. 
 \eea
 To prove  \eqref{eq:stieltjes 5.4} start with  \eqref{eq:stieltjes 5.1} as 
 \bea
 \sum_{n=0}^\infty q^{\binom{n}{2}} t^n S_n(abq^{-n};q) 
 = (abt, -t;q)_\infty 
 = (abt, -bt;q)_\infty \; \frac{(-t;q)_\infty}{(-bt;q)_\infty},
 \notag
 \eea
 then expand the first product in $S_k(aq^{-k};q)$ and the second term
  using the $q$-binomial theorem. The proof of  \eqref{eq:stieltjes 5.5}
 consists of writing $(-aq;q)_\infty/(-aq;q)_n(-aq^{n+1};q)_k$ as 
 $-aq^{n+k+1};q)_\infty$ then expand this infinite product and use 
 \eqref{eqfiniteqbinom}.  The special values in \eqref{eq:stieltjes 5.6} follow from letting $x =1$ in \eqref{eq:stieltjes 5.1} then equate like powers of $t$.  Similarly the special values in \eqref{eq:stieltjes 5.7} 
 and \eqref{eq:stieltjes 5.8} follow from putting $x = -q^{\mp 1/2}$ in 
 \eqref{eq:stieltjes 5.1}. Replace $x$ by $z$ in then multiply by 
 $(-w)^nq^{\binom{n+1}{2}}$ and sum to prove  \eqref{eq:stieltjes 5.9}. 
 To prove  \eqref{eq:stieltjes 10} we expand the right-hand side in 
 powers of $z$ and realize that the coefficient of $(-z)^n$ is 
 \bea
 \notag
 \frac{q^{n^2+mn}}{(q;q)_n} {}_2\phi_1(q^{-m}, q^{-n}; 0,; q,q).
 \eea
  By the $q$-Chu-Vandermonde sum \cite[(II.6)]{Gas:Rah} the 
  ${}_2\phi_1$  equals $q^{-mn}$. 
\end{proof}

We note that the polynomials $\{S_n(xq^{-n};q)\}$ are related to the $q^{-1}$-Hermite polynomials, \cite{Ask},  \cite{Ism:Mas}, which  
 are defined by 
 \begin{eqnarray}
 h_n{(\sinh\xi\,|\, q)} &=& \sum^n_{k=0}
 \frac{(q;q)_n}{(q;q)_k(q;q)_{n-k}}\, (-1)^k q^{k(k-n)} e^{(n-2k)\xi}.  
 \label{eqqH}
\eea
Indeed 
\bea
\label{eqSWasqH}
S_n(e^{-2\xi}q^{-n};q) = \frac{1}{(q;q)_n} h_n{(\sinh\xi\,|\, q)}. 
\eea
In fact \eqref{eq:stieltjes 5.1} is equivalent to the generating function 
for the $q^{-1}$-Hermite polynomials, \cite{Ismbook}. Moreover 
\eqref{eqSWasqH} and the generating function \cite[Theorem 21.3.1]{Ismbook}  lead to 
\bea
\sum_{n=0}^\infty \frac{(q;q)_nq^{n^2/4}}{(\sqrt{q};\sqrt{q})_n} t^n 
S_n(zq^{-n};q) = \frac{(-tq^{1/4}, -tq^{1/4} z;\sqrt{q})_\infty}
{(-t^2 z;q)_\infty}.
\eea
The 
Poisson kernel of $q^{-1}$-Hermite polynomials,  
 \cite[Theorem 21.2.3]{Ismbook} implies 
\bea
\sum_{n=0}^\infty  (q;q)_nq^{\binom{n}{2}} t^n 
S_n(zq^{-n};q)S_n(\zeta q^{-n};q) 
= \frac{(-t, -t z\zeta, t z, t \zeta;q)_\infty}{(t^2z \zeta/q;q)_\infty}. 
\eea
Similarly one can derive other generating  relations. 

It must be noted that \eqref{eq:stieltjes 5.7} and 
\eqref{eq:stieltjes 5.8} when written in terms of the $q^{-1}$-Hermite 
polynomials are the evaluation of $h_n(0|q)$, see 
\cite[Corollary 21.2.2]{Ismbook}.  It is easy to see that   
the evaluations 
\eqref{eq:stieltjes 5.7} and 
\eqref{eq:stieltjes 5.8}  are equivalent to  the identity in the following theorem.  
\begin{thm}
We have
\bea
\label{eqGFhn0}
A_{q^{2}}\left(-b^{2}\right)=\left(b\sqrt{q};q\right)_{\infty}\sum_{n=0}^{\infty}\frac{q^{n^{2}/2}b^{n}}{\left(q,b\sqrt{q};q\right)_{n}},
\eea
 \end{thm}


\end{document}